\documentclass[12pt,a4paper]{article}
\usepackage{xcolor}
\usepackage{amsmath}
\usepackage{amsthm}
\usepackage{amsfonts}
\usepackage{amssymb}
\usepackage{comment}
\usepackage{authblk}
\usepackage{framed}
\usepackage{url}
\usepackage{url}

\theoremstyle{definition}
\newtheorem{thm}{Theorem}[section]
\newtheorem*{thm'}{Theorem}

\newtheorem{lem}[thm]{Lemma}
\newtheorem{defn}[thm]{Definition}
\newtheorem{cor}[thm]{Corollary}
\newtheorem{nota}[thm]{Notation}

\newtheorem{remark}[thm]{Remark}

\newenvironment{g}{G\"{o}del~}

\begin{document}
\sloppy
\title{The Craig Interpolation Property in First-order G\"odel Logic }
\author[1]{S.M.A. Khatami\footnote{khatami@birjandut.ac.ir}}
\author[2]{M. Pourmahdian\footnote{ pourmahd@ipm.ir}}
\author[2,3]{N. R. Tavana\footnote{corresponding author, nrtavana@aut.ac.ir}}
\begin{tiny}
\affil[1]{Department of Computer Science, Birjand University of Technology, Birjand, Iran}
\affil[2]{Department of Mathematics and Computer Science, Amirkabir University of Technology (Tehran Polytechnic), Hafez Avenue 15194, P.O.Box 15875-4413, Tehran, Iran}
\affil[3]{School of Mathematics, Institute for Research in Fundamental Sciences (IPM), P.O. Box 19395-5746, Tehran, Iran}
\end{tiny}

\maketitle
\begin{abstract}
In this article, a model-theoretic approach is proposed to prove that the first-order G\"{o}del logic, $\mathbf{G}$, as well as its extension $\mathbf{G}^{\Delta }$ associated with first-order relational languages enjoy the Craig interpolation property. These results partially provide  an affirmative answer to a question posed in \cite{ag-baaz}.
\end{abstract}
\section{Introduction}
The Craig interpolation theorem, in its original form, exhibits a deep connection between syntax and semantics of first-order logic (\cite{Craig}). The Craig interpolation property has been studied in various non-classical logics, such as intuitionistic  logic (\cite{MOU},\cite{Shutte}), {\L}ukasiewicz logic (\cite{Mundici}), and some fragments/extensions of \g logic (\cite{BGG},\cite{BVV},\cite{BV},\cite{NRT}). There are different approaches in proving this property. While the proof-theoretic approach gives a constructive and algorithmic way for showing this property, the model-theoretic approach, despite its non-constructive nature, generates some widespread tools leading to Lyndon interpolation, Beth definability and joint Robinson consistency theorems (\cite{CC}). The algebraic viewpoint, on the other hand, reveals a deep connection between the Craig interpolation   property and various amalgamation properties of Lindenbaum algebras associated with the logic under consideration (\cite{Maksimova},\cite{SS}). \\

The impact of interpolation property is not restricted to pure logical studies. In fact, this property has numerous applications in computing sciences (\cite{BDGM},\cite{Bo},\cite{D6},\cite{GB},\cite{Mc}). The computing sciences and model-theoretic incentives give rise to a general categorical approach to interpolation  within the theory of institutions (\cite{D1},\cite{D2},\cite{D3},\cite{D33},\cite{D4},\cite{D5}). This point of view allows one to consider this property independent of any logic. \\
Our main contribution in this article is to prove the Craig interpolation property in first-order \g logic, whenever the languages under consideration are relational. Our result gives an affirmative answer to an open problem posed by Aguilera and Baaz in \cite{ag-baaz}, problem 2.\\
 The key idea to establish the Craig interpolation property in first-order \g logic $\mathbf{G}$, is to regard $\mathbf{G}$ as a sub-logic of \g logic with $\Delta$ operator, $\mathbf{G}^\Delta$. The advantage of  working in $\mathbf{G}^\Delta$ is twofold.  Firstly, similar to \g logic, $\mathbf{G}^\Delta$ yet admits a sound and complete axiomatization (\cite{baaz2006}) which yields some model-theoretic properties such as 1-entailment compactness. Secondly, unlike the \g logic, within $\mathbf{G}^\Delta$ one can define a unary logical connective $\thicksim$ whose semantics separates between true and not true and resembles the classical negation to a great extent. These fundamental data are essential devices to mimic the classical model-theoretic proof of the Craig interpolation theorem and establish our result in \g logic. \\

This paper is organized as follows. In section \ref{sec1}, we provide essential definitions/facts about \g logic $\mathbf{G}$ and its extension  $\mathbf{G}^\Delta$. In section 3, we prove, under the assumption that the languages under consideration are relational,  both  $\mathbf{G}$ and $\mathbf{G}^\Delta$ do satisfy the Craig interpolation property. Finally section 4 is devoted to discuss some open questions for further research.   \\

\section{Preliminaries}\label{sec1}
In this section, we review the basic concepts and facts about first-order \g logic, $\mathbf{G}$, and its extension, $\mathbf{G}^\Delta$, enriched by $\Delta$ operator. Our notations follow \cite{baaz2019} and \cite{baaz2006} in which most of the basic concepts and facts can be found. \\

Throughout this article, all first-order languages  are considered to be countable and relational, i.e. they consist of a countable set  of relation symbols of various arities together with a countable set of constant symbols.  \\
Let $\mathcal{L}=\mathcal{R}\cup \mathcal{C}$ be a first-order language. The syntax of first-order \g logic, $\mathbf{G}$, with respect to the language $\mathcal{L}$ comprises propositional logical connectives $\wedge,\vee,\rightarrow,\bot$ as well as quantifiers $\forall,\exists$. We also consider a countably infinite set $Var$ of variables. Customarily, the union of the sets $Var$ and $
\mathcal{C}$ is the set of $\mathcal{L}$-terms.
 The set of $\mathcal{L}$-formulas  in  $\mathbf{G}$ can be defined inductively by $$R(t_1,\dots,t_n)\mid \varphi\wedge \psi\mid \varphi\vee \psi\mid \varphi\rightarrow\psi\mid\bot\mid \forall x\ \varphi\mid \exists x\ \varphi,$$
 where $R$ is an $n$-array relation symbol and $t_1\dots,t_n$ are $\mathcal{L}$-terms.  \\

 The \g logic can be augmented by an additional unary connective $\Delta$. The resulting logic is denoted by $\mathbf{G}^\Delta$. The set of $\mathcal{L}$-formulas in $\mathbf{G}^\Delta$ consists of
 $$R(t_1,\dots,t_n)\mid \varphi\wedge \psi\mid \varphi\vee \psi\mid \varphi\rightarrow\psi\mid\Delta\varphi\mid\bot\mid \forall x\ \varphi\mid \exists x\ \varphi.$$

An $\mathcal{L}$-formula without any free variables is called a \emph{closed formula}. A \emph{theory} is  a set of closed formulas. Unless stated otherwise, formulas and theories are meant to be in ${\mathbf G}^\Delta$.
Further, we distinguish formulas and theories in \g logic $\mathbf{G}$, by denoting them as $\mathbf{G}$-formulas and $\mathbf{G}$-theories.\\

Other logical connectives defined as follows:
\begin{eqnarray*}
\neg\varphi&:=&\varphi\to\bot,\\
\top&:=&\neg\bot,\\
\thicksim\varphi&:=&\neg\Delta\varphi,\\
\varphi\leftrightarrow\psi&:=&(\varphi\to\psi)\wedge(\psi\to\varphi).
\end{eqnarray*}

Below, we define the notion of a valuation which serves as (standard) semantics for both $\mathbf{G}$ and $\mathbf{G}^\Delta$.

\begin{defn}
A \emph{valuation} $v$ consists of
\begin{enumerate}
  \item a nonempty set $M$, as the underlying universe (or universe),
  \item for each $k$-array relation symbol $R\in\mathcal{R}$, a function $R^v: M^k\to[0,1]$,
  \item
  for each constant symbols $c\in C$, $c^v\in M$.
  \end{enumerate}
\end{defn}

Given a valuation $v$, the value of any formula is defined in a natural way.
\begin{defn}
For a valuation $v$, the value of a formula $\varphi(\bar{x})$ with respect to a tuple $\bar{a}$ is inductively defined
as follows:
\begin{enumerate}
  \item $v(\bot)=0$,
  \item for atomic formula $R(x_1, \cdots, x_k)$, $v(R(a_1, \cdots, a_k))=R^v(a_1, \cdots, a_k)$,
  \item $v(\varphi\wedge\psi)=\min\{v(\varphi),v(\psi)\}$,
  \item $v(\varphi\vee\psi)=\max\{v(\varphi),v(\psi)\}$,
  \item $v(\varphi\to\psi)=\left\{
  \begin{array}{ll}
  1&v(\varphi)\le v(\psi)\\
  v(\psi)&v(\varphi)> v(\psi)
  \end{array}\right.$,
  \item $v(\Delta\varphi)=\left\{
  \begin{array}{cc}
    1 & v(\varphi)=1 \\
    0 & v(\varphi)<1
  \end{array}\right.$,
  \item $v(\forall x\ \varphi(x,\bar{b}))=\inf_{a\in M}\{v(\varphi(a,\bar{b}))\}$,
  \item $v(\exists x\ \varphi(x,\bar{b}))=\sup_{a\in M}\{v(\varphi(a,\bar{b}))\}$.
\end{enumerate}
\end{defn}

Note that since $[0,1]$ is a complete linear-ordered set, the interpretations of $\forall$ and $\exists$ are well-defined. One could easily check that
\begin{itemize}
 \item[] $v(\neg\varphi)=\left\{
  \begin{array}{cc}
    1 & v(\varphi)=0 \\
    0 & v(\varphi)>0
  \end{array}\right. $,
%\item[] $v(\neg\neg\varphi)=\left\{
  % \begin{array}{cc}
   % 1 & v(\varphi)>0\\
    %0 & v(\varphi)=0
  %\end{array}\right.$,
\item[] $v(\thicksim\varphi)=\left\{
  \begin{array}{cc}
    1 & v(\varphi)<1 \\
    0 & v(\varphi)=1
  \end{array}\right.$.
\end{itemize}
%*************************************************************8
\begin{remark}
While the logical connective $\neg$ distinguishes between false and not false, the logical connective $\thicksim$ separates between true and not true. Note that $\thicksim$ cannot be defined in $\mathbf{G}$. The existence of the logical connective $\thicksim$ in $\mathbf{G}^{\Delta}$ plays an essential role in the proof of Theorem \ref{CraigforG}.
\end{remark}

\begin{defn}
A valuation $v$ is a \emph{model} of a closed formula $\varphi$, if $v(\varphi)=1$. Likewise, a valuation $v$ models a theory $T$ if $v(\varphi)=1$ for each $\varphi\in T$.
\end{defn}

\begin{defn}(1-entailment)
For a theory $T$ and a closed formula $\varphi$,
$T$ 1-entails $\varphi$, denoted by  $T\Vdash\varphi$, if $v(\varphi)=1$ for each valuation $v$ which models $T$.
\end{defn}
\begin{defn}
\begin{enumerate}
\item
A theory $T$ is \emph{satisfiable} if $T\nVdash \bot$.
\item
A closed formula $\varphi$ is \emph{tautology} if $\emptyset\Vdash\varphi$. In this situation, we write $\Vdash\varphi$.
\item
A theory $T$ is called \emph{1-entailment closed} whenever for every closed formula $\theta$,  if $T\Vdash \theta$ then $\theta\in T$.
\end{enumerate}
\end{defn}
Clearly, a theory $T$ is satisfiable if and only if it has a model. Moreover, $\varphi$ is a tautology if and only if each valuation is a model of $\varphi$.\\

The following lemma demonstrates some rudimentary  properties of 1-entailment.

\begin{lem}\label{property}
Suppose $T$ is a theory, $\varphi$ and $\psi$ are two closed formulas. Then,
\begin{enumerate}
\item  $\Vdash \Delta \varphi\vee \thicksim\varphi$,
  \item  $\Vdash \Delta\varphi\to\varphi$,
  \item $\varphi\Vdash\Delta\varphi$,
 \item   $\Delta(\varphi\wedge \psi)\Vdash  \Delta(\varphi)\wedge \Delta(\psi)$,

  \item
  $\psi\Vdash \varphi$ if and only if $\thicksim \varphi\Vdash\thicksim \psi$,

  \item

  $\Vdash\Delta\thicksim\varphi\leftrightarrow\thicksim\varphi$ and
  $\Vdash\Delta\varphi\leftrightarrow\ \thicksim\thicksim\varphi$,

  \item  $\thicksim \thicksim \varphi \Vdash  \varphi$ and $ \varphi\Vdash  \thicksim \thicksim \varphi$,

  \item   $\thicksim (\varphi\rightarrow \psi)\Vdash  \psi\rightarrow \varphi$,
  \item\label{vee}
   if $T\cup\{\varphi\}\Vdash\psi$ and $T\cup\{\thicksim\varphi\}\Vdash\psi$, then $T\Vdash\psi$,
  \item \label{neg}
  if $T\Vdash \thicksim\varphi$ and $T\Vdash\thicksim\psi$ then $T\Vdash\thicksim(\varphi\vee\psi)$,
 \item\label{imp}
 $T\cup\{\varphi\}\Vdash\psi$ if and only if $T\Vdash\Delta\varphi\rightarrow\psi$.

\end{enumerate}
\end{lem}
\begin{proof}
The proofs of 1-8 are straightforward.\\
For \ref{vee}, assume that $v(\theta)=1$, for each $\theta\in T$. If $v(\varphi)=1$, then, since $T\cup\{\varphi\}\Vdash\psi$, we have that $v(\psi)=1$. On the other hand, if
$v(\varphi)< 1$, then $v(\thicksim\varphi)=1$. So, $v$ models $T\cup\{\thicksim\psi\}$. Therefore,  $T\cup\{\thicksim\varphi\}\Vdash\psi$ implies $v(\psi)=1$.\\
For \ref{neg}, assume that $v(\theta)=1$, for each $\theta\in T$.
 Then, $v(\thicksim\varphi)=1$ and $v(\thicksim\psi)=1$. So,
$v(\varphi)<1$ and $v(\psi)<1$. Therefore, $v(\varphi\vee\psi)<1$ which means that $v(\thicksim(\varphi\vee\psi))=1$. Hence,
$T\Vdash\thicksim(\varphi\vee\psi)$.\\
For \ref{imp}, see Lemma 16 in \cite{baaz2019}.

%For , assume that $\psi\Vdash \varphi$ and $v(\thicksim \varphi)=1$. Then $v(\varphi)<1$. Since, $\psi\Vdash \varphi$ we have
%$v(\psi)<1$. Therefore, $v(\thicksim \psi)=1$. Hence, $\thicksim \varphi\Vdash\thicksim \psi$.\\
%Conversely, suppose that  $\thicksim \varphi\Vdash\thicksim \psi$ and
%$v(\psi)=1$. So $v(\thicksim\psi)=0$ which besides $\thicksim \varphi\Vdash\thicksim \psi$ leads to $v(\thicksim\varphi)<1$. Therefore,
%$v(\varphi)=1$. Hence, $\psi\Vdash \varphi$.
\end{proof}
\begin{lem}\label{constant}
Assume $T$ is an $\mathcal{L}$-theory, $\sigma(x)$ and $\chi(x)$ are $\mathcal{L}$-formulas, $\theta$ is a closed formula and $c$ is a constant symbol. Then,
\begin{enumerate}

\item    $\forall x\  \sigma(x)\Vdash  \sigma(c)$.

\item    $\sigma(c)\Vdash  \exists x\  \sigma(x)$

\item    $\thicksim \forall x\  \sigma(x) \Vdash  \exists x\ \thicksim \sigma(x)$ and $\exists x\ \thicksim \sigma(x)\Vdash\thicksim \forall x\  \sigma(x) $.

\item   $\thicksim \exists  x\  \sigma(x) \Vdash  \forall  x\ \thicksim \sigma(x)$ and $\forall x\thicksim\sigma(x)\Vdash \thicksim\exists x\ \sigma(x)$.

\item   if $c$ does not occur in $T$ and $T\Vdash \sigma(c)$, then $T\Vdash \forall x\  \sigma(x)$.

\item    if  $\sigma(c)\Vdash \theta$ and $c$ does not occur in $\theta$, then   $\exists x\ \sigma(x)\Vdash \theta$.
\item
if $T\cup\{\thicksim\forall x\ \sigma(x)\}\cup\{\thicksim\sigma(c)\}\Vdash\chi(c)$  and $c$ does not occur in $T\cup\{\sigma(x),\chi(x)\}$ then $T\cup\{\thicksim\forall x\ \sigma(x)\}\Vdash \exists x\ \chi(x)$.
\end{enumerate}

\end{lem}

\begin{proof}
The proofs of 1-6 immediately follow from $1$-entailment definition.\\

For 7, first notice that by the assumption and Lemma \ref{property}(11),(6), we have $$T\cup\{\thicksim\forall x\ \sigma(x)\}\Vdash \thicksim\sigma(c)\rightarrow\chi(c).$$  Further, since $c$  does not occur in $T\cup\{\sigma(x)\}$, by item (5) above, we have $$T\cup\{\thicksim\forall x\ \sigma(x)\}\Vdash \forall x\ (\thicksim\sigma(x)\rightarrow\chi(x)).$$

Now, let $v$ be a valuation which models $T\cup\{\thicksim\forall x\ \sigma(x)\}$. Then, $v(\thicksim\forall x\ \sigma(x))=1$ implies there is an element $a$ in the universe $M$ of $v$ such that $v(\sigma(a))<1$ and $v(\thicksim\sigma(a))=1$. On the other hand,  $v(\forall x\ (\thicksim\sigma(x)\rightarrow\chi(x)))=1$. This means $v(\thicksim\sigma(b))\leq v(\chi(b))$, for each $b\in M$. So, in particular,  $v(\chi(a))=1$. Therefore, $v(\exists x\ \chi(x))=1$.
\end{proof}

Below, we state the 1-entailment compactness theorem which can be easily proved using results from \cite{baaz2019}.

%1-entailment compactness of $\mathbf{G}^\Delta$ follows from \cite[Theorem 22]{bazz2019}, \cite[Theorem 19]{bazz2019}  and \cite[Lemma 16]{bazz2019}.
\begin{thm}
(1-entailment compactness)
In $\mathbf{G}^\Delta$, for a theory $T$ and a closed formula $\varphi$, if
$T\Vdash \varphi$, then there exists a finite subset $S$ of $T$ such that $S\Vdash\varphi$.
\end{thm}
\begin{proof}
Suppose $T\Vdash\varphi$. Then, on the basis of Theorem 22 in \cite{baaz2019}, there is a finite subset $T_0$ of $T$ such that $\vdash_{H^\Delta}\Delta (\bigwedge T_0)\rightarrow\varphi$, where $\vdash_{H^\Delta}$ corresponds to the hypersequent proof system given in \cite{baaz2006}. Moreover, by soundness theorem, Theorem 19 in \cite{baaz2019}, $\Vdash \Delta(\bigwedge T_0)\rightarrow\varphi$. Therefore, by item \ref{imp} of Lemma \ref{property}, we have $T_0\Vdash\varphi$.
\end{proof}

\begin{cor}
A theory $T$ is satisfiable if and only if it is finitely satisfiable, i.e.  every finite subset $T_0$ of $T$ is satisfiable.

\end{cor}

%*************************************************************8
%Some axioms and properties related to $\Delta$
%\begin{itemize}
  %\item[$\Delta1$] $\Delta \varphi\vee \thicksim\varphi$
  %\item[$\Delta2$] $\Delta(\varphi\vee\psi)\to(\Delta\varphi\vee%\Delta\psi)$
  %\item[$\Delta3$] $\Delta\varphi\to\varphi$
  %\item[$\Delta4$] $\Delta\varphi\to\Delta\Delta\varphi$
  %\item[$\Delta5$] $\Delta(\varphi\to\psi)\to(\Delta\varphi\to\Delta\psi)$
 % \item[$\Delta6$] $\varphi\vdash\Delta\varphi$
  %\item By $\Delta3$ we have $\Sigma\Vdash \varphi$ iff $\Sigma\Vdash \Delta \varphi$.
%\end{itemize}
%*************************************************************8

%*************************************************************8

\begin{defn}
Let $\Sigma$ be a theory. Then,
\begin{enumerate}
\item  $\Sigma$ is \emph{$\thicksim$-complete} if for every closed formula $\varphi$, either $\Sigma\Vdash \varphi$ or $\Sigma\Vdash\thicksim\varphi$.
\item  $\Sigma$ is called \emph{linearly complete} if  for every closed formulas $\varphi$ and $\psi$, either $\Sigma\Vdash \varphi\to\psi$ or $\Sigma\Vdash\psi\to\varphi$.
\item  $\Sigma$  is \emph{maximally satisfiable}  if  for every closed formula $\varphi$, if $\Sigma\cup\{\varphi\}$ is satisfiable then $\Sigma\Vdash \varphi$.
\end{enumerate}
\end{defn}

The following lemma shows all the concepts given in the above definition are equivalent.
\begin{lem}\label{eqd}
Let $\Sigma$ be a theory. Then, the following conditions are equivalent.

\begin{enumerate}
\item  $\Sigma$ is  linearly complete.
\item   For each closed formulas $\varphi$ and $\psi$, if $\Sigma\Vdash\varphi\vee \psi$ then $\Sigma\Vdash\varphi$ or $\Sigma\Vdash\psi$.
\item  $\Sigma$ is $\thicksim$-complete.
\item  $\Sigma$ is a maximally satisfiable.
\end{enumerate}

\end{lem}
\begin{proof}
Without loss of generality, we may assume that $\Sigma$ is satisfiable , as otherwise, $\Sigma$ obviously fulfills all the above conditions. The equivalence of 1 and 2 follows from Lemma 5.2.3 of \cite{hajek98}.
Now, assume that $\Sigma$ is $\thicksim$-complete. We want show that $\Sigma$ is linearly complete. Assume not.  Then, for some closed formulas $\varphi$ and $\psi$, suppose that $\Sigma\nVdash \varphi\rightarrow \psi$ and $\Sigma\nVdash \psi\rightarrow\varphi$. So,  $\Sigma\Vdash \thicksim(\varphi\rightarrow \psi)$ and $\Sigma\Vdash\thicksim(\psi\rightarrow\varphi)$. Since $\Sigma$ is satisfiable, there is a valuation $v$ which models $\Sigma$. Thus,
$v(\varphi\rightarrow\psi)<1$ and $v(\psi\rightarrow \varphi)<1$. Hence, $v(\varphi)>v(\psi)$ and $v(\psi)>v(\varphi)$, which is a contradiction. Therefore, 3 implies 1.\\
Now, we show that 2 implies 3. Suppose $\Sigma$ satisfies item 2. Then, by Lemma \ref{property}(1), $\Sigma\Vdash\Delta(\varphi)\vee \thicksim \varphi$. So, $\Sigma\Vdash \Delta(\varphi)$ or $\Sigma\Vdash \thicksim \varphi$. But, the former case implies that $\Sigma\Vdash \varphi$, by Lemma \ref{property}(2). Hence, we either have $\Sigma\Vdash\varphi$ or $\Sigma\Vdash \thicksim \varphi$. Thus $\Sigma$ is $\thicksim$-complete.\\
Now, we prove that 3 implies 4. To this end, suppose $\Sigma\cup \{\varphi\}$ is satisfiable. If $\Sigma\nVdash \varphi$, then $\Sigma\Vdash \thicksim\varphi$. But, this implies that $\Sigma\cup \{\varphi\}$ is not satisfiable, a contradiction. Therefore, $\Sigma$ is maximally satisfiable.\\
Finally, to verify 4 implies 2 notice that if  $\Sigma$ is maximally satisfiable and $\Sigma\Vdash\varphi\vee\psi$ then either of $\Sigma\cup\{\varphi\}$ and  $\Sigma\cup\{\psi\}$ are satisfiable. Therefore, $\Sigma\Vdash\varphi$ or $\Sigma\Vdash\psi$.
\end{proof}
In the view of the above lemma, we say that a theory $\Sigma$ is \emph{complete} if it satisfies one of the above equivalent conditions.
In the following, we recall the concept of a Henkin theory.
\begin{defn}{

A theory $T$ is called \emph{Henkin} whenever for each formula $\varphi(x)$, if $T\nVdash \forall x\ \varphi(x)$ then there exists a constant symbol $c$ such that $T\nVdash \varphi(c)$.}

\end{defn}

 \begin{lem}
 A satisfiable $\mathcal{L}$-theory can be extended to a complete satisfiable Henkin theory $T'$ in a language $\mathcal{L}'$ extending $\mathcal{L}$.
 \end{lem}
\begin{proof}
The proof follows from techniques available in the completeness theorem. One can simply replace  the notion of consistency with (finite) satisfiability.
\end{proof}
\begin{defn}\label{simT}
Let $T$ be a complete satisfiable  $\mathcal{L}$-theory. Consider the equivalence relation $\sim_T$ on the set of closed $\mathcal{L}$-formulas defined as
\begin{center}
$\theta\sim_T\chi$ if and only if $T\Vdash \theta\leftrightarrow\chi$.
\end{center}
 Denote $\mathcal{B}_T$ to be the corresponding set of equivalence classes of the relation $\sim_T$.
\end{defn}

Since $T$ is complete and satisfiable, the relation $\leq$ defined as
 $$[\theta]\leq[\chi]\ \hbox{if and only if}\ (\theta\rightarrow\chi)\in T$$
gives a well-defined linear-order on $\mathcal{B}_T$. Further, $(\mathcal{B}_T,\leq)$ can be turned into a $\mathbf{G}^\Delta$-algebra $(\mathcal{B}_T,\wedge,\vee,\rightarrow,\Delta, \mathbf{0},\mathbf{1})$ by assuming
\begin{align*}
[\varphi]\wedge[\psi]&:=[\varphi\wedge\psi],\\
[\varphi]\vee[\psi]&:=[\varphi\vee\psi],\\
[\theta]\rightarrow[\psi]&:=[\theta\rightarrow\psi],\\
\Delta([\varphi])&:=[\Delta\varphi],\\
\mathbf{0}&:=[\bot],\\
\mathbf{1}&:=[\top].
\end{align*}
The linear-ordered set $(\mathcal{B}_T,\leq)$ plays an essential role in proving Craig interpolation property in $\mathbf{G}$ and $\mathbf{G}^\Delta$.
\begin{defn}
A satisfiable $\mathbf{G}$-theory $T$ is said to be \emph{$\mathbf{G}$-linearly complete} if for each pair of closed $\mathbf{G}$-formulas $\theta_1 $ and $\theta_2$, either $T\Vdash \theta_1\rightarrow \theta_2$ or $T\Vdash \theta_2\rightarrow \theta_1$.
\end{defn}

 If $T$ is $\mathbf{G}$-linearly complete then one  can also define $\mathcal{B}_T$ by restricting closed formulas to closed $\mathbf{G}$-formulas. Then, the same method can be applied to $\mathcal{B}_T$ to define a linear-order $\leq$ on the set of equivalence classes of $\sim_T$. Moreover, $(\mathcal{B}_T,\leq)$ can be naturally seen as a $\mathbf{G}$-algebra $(\mathcal{B}_T,\wedge,\vee,\rightarrow,\mathbf{0},\mathbf{1})$.

%\begin{defn}
%An algebra $\mathcal{A}=\langle A, \vee, \wedge, \to, 0, 1\rangle$ is a $\mathbf{G}$-algebra if
%\begin{itemize}
 % \item $\langle A, \vee, \wedge, 0, 1\rangle$ is a bounded lattice with least element $0$ and largest element $1$,
 % \item forall $a, b, c\in A$, $(a\wedge b)\le c$ if and only if $a\le b\to c$,
  %\item forall $a, b \in A$, $(a\to b)\vee(b\to a)=1$.
%\end{itemize}
%A $\mathbf{G}$-algebra $\mathcal{A}$ equipped with the unary operator $\Delta:A\to A$ defined by
%$\Delta(a)=\left\{\begin{array}{cc}
%1 & a=1\\
%0 & a\ne 1
%\end{array}\right.$ is called a $\mathbf{G}^\Delta$ algebra.
%\end{defn}

%\begin{lem}\label{embed}
%Let $\mathcal{L}_0\subseteq\mathcal{L}_1$. Suppose that $T_0\subseteq T_1$ are two maximally $\Delta$-complete Henkin theories respectively in $\mathcal{L}_0$ and $\mathcal{L}_1$. Then, the canonical model $v_0$ of $T_0$ can be embedded in the canonical model $v_1$ of $T_1$.
%\end{lem}

%\begin{remark}\label{constant}
%\begin{enumerate}
%\item
%If $\varphi(\bar c,\bar d)\Vdash \psi(\bar e,\bar d)$ then $\varphi(\bar c,\bar d)\Vdash \forall\bar x\ \psi(\bar x,\bar d)$ and
%$\forall\bar x\ \psi(\bar x,\bar d)\Vdash\psi(\bar e,\bar d)$.
%\item
%If $\varphi(\bar e,\bar d)\Vdash \psi(\bar d)$ then $\varphi(\bar e,\bar d)\Vdash\exists \bar x\ \varphi(\bar x,\bar d)$ and $\exists\bar x\ \varphi(\bar x,\bar d)\Vdash \psi(\bar d)$.
%\end{enumerate}
%\end{remark}

\begin{defn}
Let $\mathbf{Lin}$ be the class of linear-ordered sets with the least and the greatest elements.  If $\mathcal{B}=(B,\leq)\in \mathbf{Lin}$ then we denote $0_{\mathcal{B}}$ to be the least element   and $1_{\mathcal{B}}$ to be the greatest element. For $\mathcal{B}_1,\mathcal{B}_2\in\mathbf{Lin}$, a function $f:\mathcal{B}_1\rightarrow\mathcal{B}_2$ is called a $\mathbf{Lin}$-homomorphism if it is strictly monotone with $f(0_{\mathcal{B}_1})=0_{\mathcal{B}_2}$ and $f(1_{\mathcal{B}_1})=1_{\mathcal{B}_2}$. A bijective $\mathbf{Lin}$-homomorphism is called a $\mathbf{Lin}$-isomorphism.

\end{defn}

Next, we show that the class $\mathbf{Lin}$  with respect to $\mathbf{Lin}$-homomorphisms has the amalgamation property, i.e. if for every $\mathcal{B}_0,\mathcal{B}_1,\mathcal{B}_2\in\mathbf{Lin}$ and $f_1:\mathcal{B}_0\rightarrow\mathcal{B}_1$ and $f_2:\mathcal{B}_0\rightarrow\mathcal{B}_2$ are two $\mathbf{Lin}$-homomorphisms then there are
$\mathcal{B}\in\mathbf{Lin}$ and two $\mathbf{Lin}$-homomorphisms
$g_1:\mathcal{B}_1\rightarrow\mathcal{B}$ and $g_2:\mathcal{B}_2\rightarrow\mathcal{B}$ such that $g_2\circ f_2=g_1\circ f_1$.

\begin{thm}\label{amalgam}
The class $\mathbf{Lin}$ with respect to $\mathbf{Lin}$-homomorphisms has the amalgamation property.
\end{thm}
\begin{proof}
For $\mathcal{B}_0,\mathcal{B}_1,\mathcal{B}_2\in\mathbf{Lin}$, let $f_1:\mathcal{B}_0\rightarrow\mathcal{B}_1$ and $f_2:\mathcal{B}_0\rightarrow\mathcal{B}_2$ be two $\mathbf{Lin}$-homomorphisms. The goal is to define $\mathcal{B}\in\mathbf{Lin}$ such that there are two $\mathbf{Lin}$-homomorphisms  $g_1:\mathcal{B}_1\rightarrow\mathcal{B}$ and $g_2:\mathcal{B}_2\rightarrow\mathcal{B}$ such that $g_2\circ f_2=g_1\circ f_1$. \\
We first assume that $\mathcal{B}_0$ is a linear-ordered subset of both
$\mathcal{B}_1$ and $\mathcal{B}_2$ with ${0}_{\mathcal{B}_0}={0}_{\mathcal{B}_1}={0}_{\mathcal{B}_2}$ and
${1}_{\mathcal{B}_0}={1}_{\mathcal{B}_1}={1}_{\mathcal{B}_2}$.
Furthermore, $\mathcal{B}_1\cap\mathcal{B}_2=\mathcal{B}_0$.

Now, define a binary relation $\sqsubseteq$ on $\mathcal{B}_1\cup\mathcal{B}_2$ as follows. For $x,y\in \mathcal{B}_1\cup\mathcal{B}_2$,
$x\sqsubseteq y$ if and only if either of the following conditions holds:
$$
  \begin{array}{l}
     x, y\in \mathcal{B}_1\ \ \mbox{and}\ \ x\leq_{\mathcal{B}_1}y, \\
     x, y\in \mathcal{B}_2\ \ \mbox{and}\ \ x\leq_{\mathcal{B}_2}y,  \\
     x\in\mathcal{B}_1\setminus\mathcal{B}_2, y\in\mathcal{B}_2 \ \ \mbox{and} \ \ \exists\alpha\in\mathcal{B}_0\ (x\leq_{\mathcal{B}_1}\alpha\leq_{\mathcal{B}_2}y), \\
    x\in\mathcal{B}_1, y\in\mathcal{B}_2\setminus\mathcal{B}_1\ \ \mbox{and}\ \
   \exists\alpha\in\mathcal{B}_0\ (x\leq_{\mathcal{B}_2}\alpha\leq_{\mathcal{B}_1}y).
  \end{array}$$

It is easy to check that $\sqsubseteq$ is a partial-order. But,
by Zorn's Lemma, $\sqsubseteq$ can be extended to a linear-order $\leq$ on $\mathcal{B}_1\cup\mathcal{B}_2$.
Then, it is easy to see that the inclusion functions $id_{\mathcal{B}_i}=g_i:\mathcal{B}_i\rightarrow\mathcal{B}$, for $i=1,2$, give the desired  $\mathbf{Lin}$-homomorphisms.
Now, for the general case, one can find $\mathcal{B}'_1,\mathcal{B}'_2\in\mathbf{Lin}$ and $\mathbf{Lin}$-isomorphisms
$h_1:\mathcal{B}_1\to\mathcal{B}'_1$ and $h_2:\mathcal{B}_2\to\mathcal{B}'_2$ such that
$\mathcal{B}_0$ is a linear-ordered subset of $\mathcal{B}'_1$ and $\mathcal{B}'_2$,
$\mathcal{B}_0=\mathcal{B}'_1\cap\mathcal{B}'_2$, $ 0_{\mathcal{B}'_1}= 0_{\mathcal{B}'_2}= 0_{\mathcal{B}_0}$, $ 1_{\mathcal{B}'_1}= 1_{\mathcal{B}'_2}= 1_{\mathcal{B}_0}$ and $h_1\circ f_1=h_2\circ f_2=id_{\mathcal{B}_0}$.
Take $\mathcal{B}\in\mathbf{Lin}$ and $\mathbf{Lin}$-homomorphisms $g_1:\mathcal{B}'_1\to\mathcal{B}$ and
$g_2:\mathcal{B}'_2\to\mathcal{B}$ such that $g_1\circ id_{\mathcal{B}_0}=g_2\circ id_{\mathcal{B}_0}$.
Then, considering $\mathbf{Lin}$-homomorphisms  $g'_1=g_1\circ h_1$ and $g'_2=g_2\circ h_2$ implies
$g'_1\circ f_1=g_1\circ h_1\circ f_1=g_1\circ id_{\mathcal{B}_0}=g_2\circ id_{\mathcal{B}_0}=g_2\circ h_2\circ f_2=g'_2\circ f_2$. Hence, the general case is also established.
\end{proof}

Finally, the following lemma shows that for any countable $\mathcal{B}\in\mathbf{Lin}$, there exists a continuous $\mathbf{Lin}$-homomorphism $f:\mathcal{B}\rightarrow [0,1]$. Recall that for $\mathcal{B}_1,\mathcal{B}_2\in\mathbf{Lin}$, a $\mathbf{Lin}$-homorphism $f:\mathcal{B}_1\rightarrow\mathcal{B}_2$ is called \emph{continuous} if $f$ preserves all
existing $\inf$'s and $\sup$'s. More precisely, if $X\subseteq \mathcal{B}_1$ and has the least upper bound (respectively, the greatest lower bound) $\alpha$ then $f(\alpha)$ is the least upper bound (respectively, the greatest lower bound) of $f(X)$.

\begin{lem}\label{standard}\cite[Lemma 5.3.2]{hajek98}
For every countable  $\mathcal{B}\in\mathbf{Lin}$, there is a continuous $\mathbf{Lin}$-homomorphism $h:\mathcal{B}\rightarrow[0,1]$.
\end{lem}

\section{Craig Interpolation Property in $\mathbf{G}$}\label{sec2}

In this section, we prove the Craig interpolation property in first-order \g logic, $\mathbf{G}$. Although we accomplish this result in $\mathbf{G}$, to implement our model-theoretic analysis, we need to carry out our investigations within ${\mathbf G}^\Delta$ and view $\mathbf{G}$ as a sub-logic of ${\mathbf G}^\Delta$. In particular, as stated before, formulas and theories are considered within ${\mathbf G}^\Delta$. Moreover, formulas and theories inside $\mathbf{G}$ are indicated by $\mathbf{G}$-formulas and $\mathbf{G}$-theories, respectively.\\

Our key model-theoretic concept is the notion of separation of two
theories within their common languages. This notion stems in the proof of the Craig interpolation property in classical model theory  (\cite{CC}, Section 2.2).

\begin{defn}\label{separability}
Suppose $\mathcal{L}_1$ and $\mathcal{L}_2$ are two languages. Let $T$ and $U$ be two theories in $\mathcal{L}_1$ and $\mathcal{L}_2$, respectively.
\begin{enumerate}
\item   A closed formula $\theta$
 \emph{separates} $T$ and $U$ in $\mathcal{L}_1\cap\mathcal{L}_2$  , if $T\Vdash\theta$ and $U\Vdash\thicksim\theta$. $T$ and $U$ are called \emph{inseparable} if there is no such $\theta$.

\item If a closed $\mathbf{G}$-formula $\theta$ separates $T$ and $U$ then $T$ and $U$ are called \emph{$\mathbf{G}$-separable}. Otherwise, we say that  $T$ and $U$ are \emph{$\mathbf{G}$-inseparable}.
\end{enumerate}

\end{defn}

Below, we present some basic properties for both notions of inseparability.   Items 2-3 in the following lemma provide a technique to extend two $\mathbf{G}$-inseparable (inseparable) theories to complete theories yet $\mathbf{G}$-inseparable (inseparable).

\begin{lem}\label{eitheror}
Suppose  two theories $T$ in $\mathcal{L}_1$ and $U$ in $\mathcal{L}_2$ are $\mathbf{G}$-inseparable (inseparable). Then,
\begin{enumerate}
\item
both $T$ and $U$ are satisfiable.
\item

for any closed formula $\varphi$ in $\mathcal{L}_1$, either $T\cup\{\varphi\}$ and $U$ are $\mathbf{G}$-inseparable (inseparable) or $T\cup\{\thicksim \varphi\}$ and $U$ are $\mathbf{G}$-inseparable (inseparable).

\item
for any closed formula $\psi$ in $\mathcal{L}_2$, either $T$ and $U\cup\{\psi\}$ are $\mathbf{G}$-inseparable (inseparable) or $T$ and $U\cup\{\thicksim \psi\}$ are $\mathbf{G}$-inseparable (inseparable).

\end{enumerate}

\end{lem}

\begin{proof}
We only show the above statements when $T$ and $U$ are $\mathbf{G}$-inseparable.

\begin{enumerate}
\item
Suppose $T$ is not satisfiable. So, $T\Vdash\bot$ while $U\Vdash\thicksim\bot$. This implies $T$ and $U$ are $\mathbf{G}$-separable, a contradiction. On the other hand, if $U$ is not satisfiable then $T\Vdash\top$ whereas $U\Vdash\thicksim\top$. Again, this contradicts $\mathbf{G}$-inseparability of $T$ and $U$.
\item
Assume on the contrary, for some closed formula $\varphi$ in $\mathcal{L}_1$, $T\cup\{\varphi\}$ and $U$ as well as  $T\cup\{\thicksim \varphi\}$ and $U$ are $\mathbf{G}$-separable. Then, there exist two closed $\mathbf{G}$-formulas $\theta_1$ and $\theta_2$ in $\mathcal{L}_1\cap\mathcal{L}_2$ such that
\begin{align*}
T\cup\{\varphi\}\Vdash \theta_1 &, U\Vdash\thicksim\theta_1,\\
T\cup\{\thicksim \varphi\}\Vdash \theta_2 &, U\Vdash\thicksim\theta_2.
\end{align*}
Therefore, $T\cup\{\varphi\}\Vdash \theta_1\vee\theta_2$ and $T\cup\{\thicksim\varphi\}\Vdash \theta_1\vee\theta_2$ which implies $T\Vdash\theta_1\vee\theta_2$, by Lemma \ref{property}(9). Moreover,
$U\Vdash\thicksim\theta_1$ and $U\Vdash\thicksim\theta_2$ lead to $U\Vdash\thicksim(\theta_1\vee\theta_2)$, by Lemma \ref{property}(10). So, $\theta_1\vee\theta_2$ $\mathbf{G}$-separates $T$ and $U$, a contradiction.
\item
This can be proved same as the above.
\end{enumerate}
\end{proof}
Now, it is turn to show the Craig interpolation property in $\mathbf{G}$. First, we introduce the following notation.

\begin{nota}
For a closed formula $\varphi$, let the language $\mathcal{L}_\varphi=\mathcal{R}_\varphi\cup C_\varphi$ where $\mathcal{R}_\varphi$ and $C_\varphi$ are respectively the sets of relation and constant symbols occur in $\varphi$.
 \end{nota}

\begin{thm}\label{CraigforG}(Craig interpolation property in ${\mathbf G}$) Suppose $\varphi$ and $\psi$ are two closed $\mathbf{G}$-formulas with $\varphi\Vdash\psi$. Then, there exists a closed $\mathbf{G}$-formula $\theta$ in $\mathcal{L}_\varphi\cap\mathcal{L}_\psi$ such that $\varphi\Vdash \theta$ and $\theta\Vdash\psi$.

\end{thm}

%Before the proof, first, the following lemma about inseparability of $\varphi$ and $\thicksim\psi$ by the above assumption is given.
%\begin{lem}\label{sep}
%If two closed formulas $\varphi$ and $\psi$ have no interpolant then $\{\varphi\}$ and $\{\thicksim\psi\}$ are inseparable.
%\end{lem}
%\begin{proof}
%Suppose not. Then, there is $\theta(c_1,\dots,c_n)$ such that $%\varphi\Vdash\theta(c_1,\dots,c_n)$ and $\thicksim\psi\Vdash\thicksim\theta(c_1,\dots,c_n)$. So, for every valuation $v$, if $v(\thicksim\psi)=1$ then $v(\thicksim\theta(c_1,\dots,c_n))=1$. Therefore, for every valuation $v$, $v(\theta(c_1,\dots,c_n))=1$ implies $v(\thicksim\theta(c_1,\dots,c_n))=0$ and then, $v(\thicksim\psi)<1$. Thus, $v(\thicksim\psi)=0$ and then $v(\psi)=1$. The above argument yields that $\theta(c_1,\dots,c_n)\Vdash \psi$. So, by assuming variables $x_1,\dots,x_n$ which do not occur in $\theta(c_1,\dots,c_n)$, one can have $\varphi\Vdash \forall x_1,\dots,x_n\ \theta(x_1,\dots,x_n)$ and $\forall x_1,\dots,x_n\ \theta(x_1,\dots,x_n)\Vdash \psi$. It is a contradiction with no interpolation.
%\end{proof}

\begin{proof}
Suppose on the contrary that $\varphi\Vdash\psi$, but there is no closed $\mathbf{G}$-formula $\theta$ which satisfies the conclusion of the theorem. Let
$\mathcal{L}_0=\mathcal{L}_\varphi\cap\mathcal{L}_\psi$ and $\mathcal{L}=\mathcal{L}_\varphi\cup\mathcal{L}_\psi$.  So, $\mathcal{L}_0=C_0\cup \mathcal{R}_0$ where $C_0=C_\varphi\cap C_\psi$ and $\mathcal {R}_0=\mathcal {R}_\varphi\cap\mathcal {R}_\psi$. Now, we consider two cases:\\

$\mathbf{Case\ 1:}$ $C_0=C_{\varphi}=C_{\psi}$.

Consider $C$ to be a countably infinite set of new constant symbols, and let $\mathcal{L}'_\varphi=\mathcal{L}_\varphi\cup C$,
$\mathcal{L}'_\psi=\mathcal{L}_\psi\cup C$, $\mathcal{L}'_0=\mathcal{L}_0\cup C$, and
$\mathcal{L}'=\mathcal{L}\cup C$. \\

Note that if we consider $\mathcal{L}'_\varphi$-theory $T_0=\{\varphi\}$ and $\mathcal{L}'_\psi$-theory $U_0=\{\thicksim\psi\}$ then they are $\mathbf{G}$-inseparable. Suppose not. Then, there is a closed $\mathbf{G}$-formula $\theta(c_1,\dots,c_n)$ in $\mathcal{L}'_0$ such that $\varphi\Vdash\theta(c_1,\dots,c_n)$ and $\thicksim\psi\Vdash\thicksim\theta(c_1,\dots,c_n)$. Thus, by Lemma \ref{property}(5), $\theta(c_1,\dots,c_n)\Vdash \psi$. Now, since $c_1,\dots,c_n$ do not occur in $\varphi$, by Lemma \ref{constant}(5),  we can conclude $\varphi\Vdash \forall x_1,\dots,x_n\ \theta(x_1,\dots,x_n)$. But, $\forall x_1,\dots,x_n\ \theta(x_1,\dots,x_n)\Vdash \theta(c_1,\dots,c_n)$ and $\theta(c_1,\dots,c_n)\Vdash\psi$. Hence, $\forall x_1,\dots,x_n\ \theta(x_1,\dots,x_n)\Vdash\psi$ and $\varphi\Vdash \forall x_1,\dots,x_n\ \theta(x_1,\dots,x_n)$ which contradicts our assumption.\\

Now, let $\varphi_1,\varphi_2,\dots$ be an enumeration of all closed formulas in $\mathcal{L}'_\varphi$ and respectively,
$\psi_1,\psi_2,\cdots$ be an enumeration of all closed formulas in $\mathcal{L}'_\psi$.\\

Define two increasing sequences of theories $\{T_m\}_{m\in\mathbb{N}}$ in $\mathcal{L}'_\varphi$ and $\{U_m\}_{m\in\mathbb{N}}$ in $\mathcal{L}'_\psi$
\begin{center}
$\{\varphi\}=T_0\subseteq T_1\subseteq T_2\subseteq \dots$ \\
$\{\thicksim\psi\}=U_0\subseteq U_1\subseteq U_2\subseteq \dots$.
\end{center}
such that for every $m\in\mathbb{N}$, $T_m$ and $U_m$ have the following properties:

\begin{enumerate}
\item
$T_m$ and $U_m$ are $\mathbf{G}$-inseparable finite sets of closed formulas.

\item For each $m$, either  $\varphi_m\in T_{m+1}$ or  $\thicksim \varphi_m\in T_{m+1}$.

\item For each $m$, either  $\psi_m\in U_{m+1}$ or $\thicksim\psi_m\in U_{m+1}$.
\item   If $\varphi_m=\forall x\ \sigma(x)\notin T_{m+1}$ then there exists $c\in C$ such that $\thicksim\sigma(c)\in T_{m+1}$.

\item If $\psi_m=\forall x\ \chi(x)\notin U_{m+1}$ then there exists $d\in C$ such that $\thicksim\chi(d)\in U_{m+1}$.
%  Note that $c$ and $d$ not used before.
\end{enumerate}

Suppose for $m\in\mathbb{N}$, the theories $T_m$ and $U_m$ are already constructed. Then, Lemma  \ref{eitheror}(2) guaranties either  $T_m\cup\{\varphi_m\}$ and $U_m$ are $\mathbf{G}$-inseparable or $T_m\cup\{\thicksim \varphi_m\}$ and $U_m$ are $\mathbf{G}$-inseparable. In the former case, we let $T_{m+1}= T_m\cup\{\varphi_m\}$. Otherwise, unless $\varphi_m=\forall x\ \sigma(x)$, we let $T_{m+1}= T_m\cup\{\thicksim \varphi_m\}$. Now, suppose  $\varphi_m=\forall x\ \sigma(x)$. Also, $T_m\cup \{\thicksim \forall x\ \sigma(x)\}$ and $U_m$ are $G$-inseparable. Then, if $c\in C$ is a constant symbol which does not occur in $T_m\cup U_m\cup \{\varphi_m\}$ then set $T_{m+1}=T_m\cup\{\thicksim \forall  x\ \sigma(x)\}\cup\{\thicksim\sigma(c)\}$.  We show that $T_{m+1}$ and $U_m$ are $\mathbf{G}$-inseparable. Suppose not. Then, there is a closed $\mathbf{G}$-formula $\theta(c)$ such that
$T_m\cup\{\thicksim \forall  x\ \sigma(x)\}\cup\{\thicksim\sigma(c)\}\Vdash\theta(c)$ and $U_m\Vdash\thicksim\theta(c)$. By Lemma \ref{constant}(7) and (5), $T_m\cup\{\thicksim \forall  x\ \sigma(x)\}\Vdash \exists x\ \theta(x)$ and $U_m\Vdash \forall x\ \thicksim\theta(x)$. So $U_m\Vdash \thicksim\exists x\ \theta(x)$,  by Lemma \ref{constant}(4). This contradicts the fact that $T_m\cup\{\thicksim\forall x\ \sigma(x)\}$ and $U_m$ are $\mathbf{G}$-inseparable.

 $U_{m+1}$ can be defined similarly to retain  conditions $1-5$.\\

Set $T_\omega=\bigcup_{m<\omega}T_m$ and $U_\omega=\bigcup_{m<\omega}U_m$.  \\
 Subsequently, define
$$T^{\prime}_{\omega}=\{\theta\in \mathcal{L}'_\varphi\ : \theta \ \hbox{is a $\mathbf{G}$-formula and\ }  T_\omega \Vdash \theta\},$$
and
$$U^{\prime}_{\omega}=\{\theta\in \mathcal{L}'_\psi\ : \theta \ \hbox{is a $\mathbf{G}$-formula and\ }  U_\omega \Vdash \theta\}.$$
Then, the following properties hold:
\begin{enumerate}

\item   $T_\omega$ and $U_\omega$ are $\mathbf{G}$-inseparable. Hence, both of them are satisfiable. Furthermore, they are $1$-entailment closed.

{\em Proof.} Since $T_m$ and $U_m$ are $\mathbf{G}$-inseparable, for each $m$, the pair $T_\omega$ and $U_\omega$ are also $\mathbf{G}$-inseparable, by 1-entailment compactness. We only show that $T_{\omega}$ is $1$-entailment closed. Suppose $T_{\omega}\Vdash \varphi$ and  $\varphi\notin T_{\omega}$.  Then, by condition 2 above, $\thicksim \varphi\in T_{\omega}$. But, this implies that $T_{\omega}$ is not satisfiable, a contradiction.

\item
$T^{\prime}_{\omega}$ and $U^{\prime}_{\omega}$ are $\mathbf{G}$-inseparable. Hence, both are satisfiable.

{\em Proof.}
Since both $T_\omega$ and $U_\omega$ are $1$-entailment closed, we have that $T^{\prime}_{\omega}\subseteq T_\omega$ and $U^{\prime}_{\omega}\subseteq U_\omega$. Thus, $T'_\omega$ and $U'_\omega$ are  $\mathbf{G}$-inseparable.

\item
$T^{\prime}_\omega$  and $U^{\prime}_\omega$ are $\mathbf{G}$-linearly complete Henkin $\mathbf{G}$-theories in $\mathcal{L}'_\varphi$ and $\mathcal{L}'_\psi$, respectively.

{\em Proof.}
We only show the claim for $T^{\prime}_\omega$.  To show $\mathbf{G}$-linearly completeness, suppose
$T^{\prime}_\omega\nVdash\varphi_m=\theta_1\rightarrow \theta_2$, for two $\mathbf{G}$-formulas $\theta_1$ and $\theta_2$. So, $\varphi_m\notin T^{\prime}_\omega$ and  thus, $\thicksim\varphi_m\in T_\omega$. But, Lemma \ref{property}(8) implies  $\thicksim (\theta_1\rightarrow \theta_2) \Vdash \theta_2\rightarrow \theta_1$ . Hence, $\theta_2\rightarrow \theta_1\in T^{\prime}_{\omega}$. \\
 Finally, notice that if $T^{\prime}_{\omega}\nVdash\forall x\ \sigma(x)$ then $\forall x\ \sigma(x)\notin T_\omega$. Hence, by condition 4, there exists $c\in C$ such that $\thicksim \sigma(c)\in T_\omega$. Thus, $T'_\omega\nVdash \sigma(c)$. Therefore,  $T^{\prime}_{\omega}$ is Henkin.
\item
$T^{\prime}_\omega\cap U^{\prime}_\omega$ is a satisfiable $\mathbf{G}$-linearly complete Henkin $\mathbf{G}$-theory in $\mathcal{L}'_0$.

{\em Proof.}
Satisfiability of $T^{\prime}_\omega\cap U^{\prime}_\omega$ is obvious. Let $\theta_1$ and $\theta_2$ be two closed $\mathbf{G}$-formulas in $\mathcal{L}'_0$ and $\sigma=(\theta_1\rightarrow \theta_2)$. Then, by the construction of $T_\omega$ and $U_\omega$,
$$\sigma\in T_\omega\ \hbox{or}\ \thicksim\sigma\in T_\omega,$$
and
$$\sigma\in U_\omega\ \hbox{or}\ \thicksim\sigma\in U_\omega.$$
Moreover, by $\mathbf{G}$-inseparability of $T_\omega$ and $U_\omega$, it is not possible either, $\sigma\in T_\omega$ and $\thicksim\sigma\in U_\omega$, or,
$\sigma\in U_\omega$ and $\thicksim\sigma\in T_\omega$. Hence, either $(\theta_1\rightarrow \theta_2) \in T^{\prime}_\omega\cap U^{\prime}_\omega$ or $(\theta_2\rightarrow \theta_1)\in T^{\prime}_\omega\cap U^{\prime}_\omega$. Also, it is clear that $T^{\prime}_\omega\cap U^{\prime}_\omega$ is Henkin.
\end{enumerate}
Now, let $\mathcal{B}_0=\mathcal{B}_{T^{\prime}_\omega\cap U^{\prime}_\omega}$,  $\mathcal{B}_1=\mathcal{B}_{T^{\prime}_\omega}$ and  $\mathcal{B}_2=\mathcal{B}_{U^{\prime}_\omega}$ be the countable linear-ordered sets  introduced in Definition \ref{simT}.
So, $\mathcal{B}_0$, $\mathcal{B}_1$, and $\mathcal{B}_2$ consist of equivalence classes with respect to equivalence relations $\sim_{T^{\prime}_\omega\cap U^{\prime}_\omega}$, $\sim_{T^{\prime}_\omega}$, and $\sim_{U^{\prime}_\omega}$, respectively. Denote the equivalence classes in
$\mathcal{B}_0$, $\mathcal{B}_1$ and $\mathcal{B}_2$  by $[.]_0$, $[.]_\varphi$ and
$[.]_\psi$, respectively. Then, there are two $\mathbf{Lin}$-homomorphisms $f_1:\mathcal{B}_0\rightarrow \mathcal{B}_1$ and $f_2:\mathcal{B}_0\rightarrow \mathcal{B}_2$ defined as $f_1([\theta]_0)=[\theta]_\varphi$ and $f_2([\theta]_0)=[\theta]_\psi$. By Theorem \ref{amalgam},  the class of $\mathbf{Lin}$ with $\mathbf{Lin}$-homomorphisms has the amalgamation property. Thus,  there exist a countable  $\mathcal{B}\in \mathbf{Lin}$ and two $\mathbf{Lin}$-homomorphisms $g_1:\mathcal{B}_1\rightarrow \mathcal{B}$ and  $g_2:\mathcal{B}_2\rightarrow \mathcal{B}$  with   $g_2\circ f_2=g_1\circ f_1$.  Moreover, on the basis of Lemma \ref{standard}, there exists a continuous  $\mathbf{Lin}$-homomorphism  $h: \mathcal{B}\rightarrow [0,1]$.\\

Define a standard $\mathcal{L}'$-valuation $v=(M,\dots)$ as follows:

%Let $v_0=(M_0,\dots)$, $v_1=(M_1,\dots)$ and $v_2=(M_2,\dots)$ be the canonical models of $T_\omega\cap U_\omega$, $T_\omega$ and $U_\omega$, respectively. Then, by Lemma %ref{canonical}, $M_0=M_1=M_2=C_0\cup C$. On the basis of Lemma
%\ref{embed}, there are embeddings $H_1:v_0\rightarrow v_1$ and $H_2:v_0\rightarrow v_2$. By Lemma \ref{standard}, there are  the standard valuations $v'_1=(M_1,\dots)$ and %$v'_2=(M_2,\dots)$ associated with $v_1$ and $v_2$, respectively. Now we define a valuation $v=(M,\dots)$ in the  language $\mathcal{L}'$ as follows.
\begin{enumerate}
\item
$M=C\cup C_0$.
\item\label{relation1}
For every $n$-array relation symbol $R\in \mathcal {R}_\varphi$ and $\bar d\in M^n$, $$R^v(\bar{d})=h(g_1([R(\bar{d})]_\varphi)).$$
\item\label{relation2}
For every $n$-array relation symbol $R\in \mathcal {R}_\psi$ and $\bar d\in M^n$, $$R^v(\bar{d})=h(g_2([R(\bar{d})]_\psi)).$$

\item For every constant symbols $c\in M$, $c^v=c$.
\end{enumerate}

Notice that if $R\in\mathcal{R}_\varphi\cap\mathcal{R}_\psi$ then
$g_1([R(\bar {d})]_\varphi)=g_1\circ f_1([R(\bar d)]_0)=g_2\circ f_2([R(\bar d)]_0)=g_2([R(\bar d)]_\psi)$. So, the above definition is well-defined.
We claim that
\begin{enumerate}
\item
 for every $\mathbf{G}$-formula $\theta(x_1,\dots,x_k)\in \mathcal {L}_\varphi$ and $\bar d\in M^k$, $v(\theta(\bar{d}))=h(g_1([\theta(\bar{d})]_\varphi))$,
 \item

 for every $\mathbf{G}$-formula $\theta(x_1,\dots,x_k)\in \mathcal {L}_\psi$ and $\bar d\in M^k$, $v(\theta(\bar{d}))=h(g_2([\theta(\bar{d})]_\psi))$.

 \end{enumerate}
The above claims particularly  imply for  every $\mathbf{G}$-formula $\theta(x_1,\dots,x_k)\in \mathcal {L}_0$ and $\bar d\in M^k$, $v(\theta(\bar{d}))=h(g_1\circ f_1([\theta(\bar{d})]_0))=h(g_2\circ f_2([\theta(\bar{d})]_0))$.
Both claims can be shown by induction on the complexity of formulas. We only show the first claim.
\begin{enumerate}
\item
Induction base for atomic formulas as well as induction steps for $\wedge$, $\vee$, $\to$ and $\Delta$ use the above definitions and the fact that $h$ is a strictly monotone function with
$h([\top]_\varphi)=1$  and we leave them for readers to verify.
\item
Suppose a $\mathbf{G}$-formula of the form $\theta(x_1,\dots,x_k):=\exists y\  \psi(y,x_1,\dots,x_k)$. Now, since $h$ is a strictly monotone continuous function, we have for $\bar d\in M^k$,
\begin{eqnarray*}
v(\theta(\bar{d}))&=&\sup_{e\in M}\ v(\psi(e,\bar{d}))=\sup_{e\in M} \ h([\psi(e,\bar{d})]_\varphi)\\
&=&h(\sup_{e\in M}\ [\psi(e,\bar{d})]_\varphi)=h([\exists x\ \psi(x,\bar{d})]_\varphi).
\end{eqnarray*}
\end{enumerate}

The above claim particularly implies $v(\varphi)=h([\varphi]_\varphi)=1$ while $v(\psi)=h([\psi]_\psi)<1$. But, this contradicts the fact that $\varphi\Vdash\psi$.\\

$\mathbf{Case\  2}$:  For the general case, assume $C_0=\{c_1,\dots,c_n\}$, $C_\varphi=C_0\cup\{d_1,\dots,d_k\}$ and
$C_\psi=C_0\cup\{e_1,\dots,e_l\}$. Then,

$$\varphi(\bar d, \bar c)\Vdash \psi(\bar e,\bar c).$$

In the light of Lemma \ref{constant}(5), since $\bar{e}$ do not occur in $\varphi(\bar d, \bar c)$, we have

$$\varphi(\bar d, \bar c)\Vdash \forall \bar{y}\ \psi(\bar{y},\bar c).$$
Again, by  Lemma \ref{constant}(6), since $\bar{d}$ do not occur in $\forall \bar y\ \psi(\bar y,\bar c)$, it follows that

$$\exists\bar x\ \varphi(\bar x, \bar c)\Vdash \forall \bar y\ \psi(\bar y,\bar c).$$
Hence, by Case 1, there exists a closed $\mathbf{G}$-formula $\theta$ in $\mathcal{L}_0$ such that
$\exists\bar x\ \varphi(\bar x, \bar c)\Vdash \theta$ and $\theta\Vdash \forall \bar{y}\ \psi(\bar y,\bar c)$. So, $\varphi\Vdash \theta$ and $\theta\Vdash \psi$.
\end{proof}

\begin{remark}
One can easily adapt the proof of Theorem \ref{CraigforG} to show that the Craig interpolation property holds in $\mathbf{G}^{\Delta}$. To this end, one should first replace the notion of $\mathbf{G}$-inseparability with inseparability, Definition \ref{separability}(1). Notice that both notions of inseparability share the same properties given in Lemma \ref{eitheror}.

\end{remark}

\begin{thm}\label{craigdelta}(Craig interpolation in ${\mathbf G}^\Delta$) Suppose $\varphi$ and $\psi$ are two closed formulas in the logic $\mathbf{G}^\Delta$. If $\varphi\Vdash\psi$ then there exists a closed formula $\theta$ in $\mathcal{L}_\varphi\cap\mathcal{L}_\psi$ such that $\varphi\Vdash \theta$ and $\theta\Vdash\psi$.

\end{thm}

\begin{proof}
One could follow the same path as in the proof of Theorem \ref{CraigforG} to define $T_{\omega}$ and $U_{\omega}$. Moreover, the same proof applies to show that theories  $T_{\omega}$ and  $U_{\omega}$  are inseparable and therefore, satisfiable. Further, $T_{\omega}$, $U_{\omega}$ and $T_{\omega}\cap U_{\omega}$ are complete  Henkin theories in $\mathcal{L}'_\varphi$, $\mathcal{L}'_\psi$ and $\mathcal{L}'_0$, respectively. Finally, definition of the valuation $v$ and its properties remain the same as the above.
\end{proof}
\begin{remark}\label{dedent}

It is known that in $\mathbf{G}$, the $1$-entailment relation coincides with the corresponding notions of deduction relation $\vdash$ and entailment relation $\models$, (\cite{baaz2019}, Defenition 13). Hence, the deductive and entailment versions of the Craig interpolation also hold in $\mathbf{G}$. On the other hand, in $\mathbf{G}^\Delta$, the $1$-entailment relation agrees with the notion of $\vdash_{H^{\Delta}}$ defined through hypersequent proof system,  \cite{baaz2006}. So, the deductive version of the Craig interpolation property is also valid in  $\mathbf{G}^\Delta$.

\end{remark}

\section{Conclusions and Further Works}\label{sec3}
 In this article, we showed that for each pair of closed $\mathbf{G}$-formulas $\varphi$ and $\psi$ whose languages consist of constant and relation symbols, if $\varphi\Vdash \psi$ then there is a closed $\mathbf{G}$-formula $\theta$ in the common language of $\varphi$ and $\psi$ such that $\varphi\Vdash \theta$ and $\theta\Vdash \psi$.  The first natural problem is whether the Craig interpolation property holds for languages which consist of function symbols. With the proof terminologies of Theorem \ref{CraigforG} in mind, notice that, in the classical setting, the languages contain the equality symbol. So, the Henkin method in the classical situation allows to "represent" each term in $\mathcal{L}_{\varphi}$ and $\mathcal{L}_{\psi}$ by a constant symbol in the set $C$. On the other hand, in the absence of equality (or similarity) relation terms in $\mathcal{L}_{\varphi}$ and $\mathcal{L}_{\psi}$ have interpretations independent of the set $C$ and therefore, the last step of the proof in defining the valuation $v$ with the desired properties fails. Another interesting topics to investigate are the Lyndon interpolation, Beth definability and Robinson joint consistency in both $\mathbf{G}$ and $\mathbf{G}^{\Delta}$. It also seems plausible that our present methods can be naturally extended to \g Logic with similarity relation, possibly for languages with function symbols. Finally, in $\mathbf{G}^{\Delta}$, the entailment relation $\models$  is different from the $1$-entailment and therefore, from $\vdash_{H^{\Delta}}$. Hence, the question which remains unanswered is whether the entailment version of the Craig interpolation is valid in $\mathbf{G}^{\Delta}$.

\end{document}